\documentclass[11pt, a4paper]{amsart}
\usepackage{amsfonts,amsmath,amscd,latexsym,amssymb,amsthm}
\usepackage[english]{babel}
\usepackage[utf8]{inputenc}
\usepackage{srcltx} 
\usepackage[all]{xy}
\usepackage{calrsfs}
\usepackage{graphicx}
\usepackage{color}
\usepackage[size=2em,nohug,midshaft,scriptlabels,PostScript=dvips]{diagrams}

\usepackage[T1]{fontenc}

\newtheorem{thm}{Theorem}[section]
\newtheorem{thmx}{Theorem}
\newtheorem{prop}[thm]{Proposition}
\newtheorem{cor}[thm]{Corollary}

\newtheorem{lem}[thm]{Lemma}

\theoremstyle{definition}
\newtheorem{deff}[thm]{Definition}

\newcommand{\sheaf}[1]{\mathcal{#1}}

\renewcommand{\O}{\sheaf{O}}

\newcommand{\PP}{\mathbb{P}} 
\newcommand{\into}{\hookrightarrow} 
\newcommand{\iso}{\cong} 
\newcommand{\st}{\;\vline\;} 
\newcommand{\res}[2]{\left.#1\right|_{#2}} 
\DeclareMathOperator{\Jac}{Jac} 
\DeclareMathOperator{\Hilb}{Hilb} 

\newcommand{\ppav}{\textsc{ppav}} 

\title{Schottky via the punctual Hilbert scheme}

\author[M. G. Gulbrandsen]{Martin G. Gulbrandsen}
\address{University of Stavanger, Department of Mathematics and Natural Sciences, NO-4036 Stavanger, Norway}
\email{martin.gulbrandsen@uis.no}

\author[M. Lahoz]{Mart\'{\i} Lahoz}
\address{ Institut de Math\'ematiques Jussieu -- Paris Rive Gauche.
Universit\'e Paris Diderot -- Paris 7.
B\^{a}timent Sophie-Germain. Case 7012.
F-75205 Paris Cedex 13, France.}
\email{marti.lahoz@imj-prg.fr}

\begin{document}

\begin{abstract}
We show that a smooth projective curve of genus $g$ can be reconstructed from its polarized Jacobian
$(X, \Theta)$ as a certain locus in the Hilbert scheme $\Hilb^d(X)$, for $d=3$ and for $d=g+2$,
defined by geometric conditions in terms of the polarization $\Theta$.  The
result is an application of the Gunning--Welters trisecant criterion and the
Castelnuovo--Schottky theorem by Pareschi--Popa and Grushevsky, and its scheme
theoretic extension by the authors.
\end{abstract}

\maketitle

\section{Introduction}

Let $(X, \Theta)$ be an indecomposable principally polarized abelian variety (\ppav) of
dimension $g$ over an algebraically closed field $k$ of characteristic
different from $2$.
The polarization $\Theta$ is considered as a divisor class under algebraic equivalence,
but for notational convenience, we shall fix a representative $\Theta\subset X$.
$(X, \Theta)$ being indecomposable means that $\Theta$ is irreducible.

The geometric Schottky problem asks for geometric conditions on $(X, \Theta)$ which determine whether
it is isomorphic, as a \ppav, to the Jacobian of a nonsingular genus $g$ curve $C$.
The Torelli theorem then guarantees the uniqueness of the curve $C$ up to isomorphism.
One may ask for a constructive version: can you ``write down'' the curve $C$, starting from
$(X,\Theta)$? Even though one may embed $C$ in its Jacobian $X$, there is no canonical choice
of such an embedding, so one cannot reconstruct $C$ as a curve in $X$ without making some choices along the way.
We refer to Mumford's classic \cite{Mcurves}
for various approaches and answers to the Schottky and Torelli problems,
and also to Beauville \cite{Bschottky} and Debarre \cite{Dschottky} for more recent results.

In this note, we show that any curve $C$ sits naturally inside the punctual Hilbert scheme
of its Jacobian $X$. We give two versions: firstly, using the Gunning--Welters criterion \cite{gunning, Wflexes}, characterizing Jacobians by having many trisecants, we reconstruct $C$ as a locus in $\Hilb^3(X)$. Secondly,
using the Castelnuovo--Schottky theorem, quoted below, we reconstruct $C$ as a locus in $\Hilb^{g+2}(X)$.
In fact, for any indecomposable \ppav{} $(X, \Theta)$, we define a certain locus
in the Hilbert scheme $\Hilb^d(X)$ for $d\ge 3$, 
and show that this locus is 
either empty, or one or two copies of a curve $C$, according to whether $(X, \Theta)$
is not a Jacobian, or the Jacobian of the hyperelliptic or nonhyperelliptic curve $C$.
Then we characterize the locus in question for $d=3$ in terms of trisecants, and for $d=g+2$
in terms of being in special position with respect to $2\Theta$-translates.

To state the results precisely, we introduce some notation.
For any subscheme $V\subset X$, we shall write $V_x\subset X$ for the translate
$V-x$ by $x\in X$.
Let $\psi\colon X \to \PP^{2^{g-1}}$ be the (Kummer) map given by the linear system $|2\Theta|$.

\begin{thmx}\label{thm:schottky-A}
Let $Y\subset \Hilb^3(X)$ be the subset consisting of all subschemes $\Gamma\subset X$
with support $\{0\}$, with the property that
\begin{equation*}
\{x\in X \st \text{$\Gamma_x \subset \psi^{-1}(\ell)$ for some line $\ell\subset \PP^{2^{g-1}}$}\}
\end{equation*}
has positive dimension. Then $Y$ is closed and
\begin{enumerate}
\item if $X$ is not a Jacobian, then $Y=\emptyset$.
\item if $X\cong \Jac(C)$ for a hyperelliptic curve $C$, then
$Y$ is isomorphic to the curve $C$.
\item if $X\cong \Jac(C)$ for a non-hyperelliptic curve $C$, then $Y$ is isomorphic to a disjoint union
of two copies of $C$.
\end{enumerate}
\end{thmx}

The proof is by reduction to the Gunning--Welters criterion; more precisely to the characterization
of Jacobians by inflectional trisecants. Note that the criterion defining $Y$ only depends on
the algebraic equivalence class of $\Theta$, and not the chosen divisor.

For the second version, we need some further terminology from \cite{PPschottky} and \cite{GL}.

\begin{deff}
A finite subscheme $\Gamma\subset X$ of degree at least $g+1$ is
\emph{theta-general} if, for all subschemes $\Gamma_{d} \subset \Gamma_{d+1}$ in $\Gamma$
of degree $d$ and $d+1$ respectively, with $d\le g$, there exists $x\in X$ such that the
translate $\Theta_x$ contains $\Gamma_d$, but not $\Gamma_{d+1}$.
\end{deff}

\begin{deff}
A finite subscheme $\Gamma\subset X$ is \emph{in special position} with respect to $2\Theta$-translates if
the codimension of $H^0(X, \sheaf{I}_\Gamma(2\Theta_x))$ in $H^0(\O_X(2\Theta_x))$
is smaller than $\deg \Gamma$ for all $x\in X$.
\end{deff}

Again note that these conditions depend only on the algebraic equivalence class of
$\Theta$. The term ``special position'' makes most sense for
$\Gamma$ of small degree, at least not exceeding $\dim H^0(\O_X(2\Theta_x)) =
2^g$.

Our second version reads:

\begin{thmx}\label{thm:schottky-B}
Let $Y\subset \Hilb^{g+2}(X)$ be the
subset consisting of all subschemes $\Gamma\subset X$ with support $\{0\}$,
which are theta-general and in special position with respect to $2\Theta$-translates.
Then $Y$ is locally closed, and
\begin{enumerate}
\item if $X$ is not a Jacobian, then $Y=\emptyset$.
\item if $X\cong \Jac(C)$ for a hyperelliptic curve $C$, then
$Y$ is isomorphic to the curve $C$ minus its Weierstra\ss \ points.
\item if $X\cong \Jac(C)$  for a non-hyperelliptic curve $C$, then $Y$ is isomorphic to a disjoint union
of two copies of $C$ minus its Weierstra\ss \ points.
\end{enumerate}
\end{thmx}

The proof of Theorem \ref{thm:schottky-B} is by reduction to the Castelnuovo--Schottky theorem, which is the following:

\begin{thm}\label{thm:castelnuovo-schottky}
Let $\Gamma\subset X$ be a finite subscheme of degree $g+2$,
in special position with respect to $2\Theta$-translates,
but theta-general. Then there exist a nonsingular curve $C$ and an isomorphism
$\Jac(C) \cong X$ of \ppav{}s, such that $\Gamma$ is contained in the
image of $C$ under an Abel--Jacobi embedding.
\end{thm}

Here, an \emph{Abel--Jacobi embedding} means a map $C\to \Jac(C)$ of the form $p\mapsto p-p_0$
for some chosen base point $p_0\in C$.
This theorem, for reduced $\Gamma$, is due to Pareschi--Popa \cite{PPschottky} and,
under a different genericity hypothesis, Grushevsky \cite{grushevsky, grushevsky-err}.
The scheme theoretic extension stated above is by the authors \cite{GL}.
The scheme theoretic generality is clearly essential for the application
in Theorem \ref{thm:schottky-B}.

We point out that the Gunning--Welters criterion is again the fundamental result
that underpins Theorem \ref{thm:castelnuovo-schottky}, and
thus Theorem \ref{thm:schottky-B}.
More recently, Krichever \cite{krichever} showed that Jacobians are in fact
characterized by the presence of a single trisecant (as opposed to a positive
dimensional family of translations), but we are not making use of this result.

\section{Subschemes of Abel--Jacobi curves}

For each integer $d\ge 1$, let
\begin{equation*}
Y_d\subset \Hilb^d(X)
\end{equation*}
be the closed subset consisting of all degree $d$ subschemes $\Gamma\subset X$ such that
\begin{enumerate}
\item[(i)] the support of $\Gamma$ is the origin $0\in X$,
\item[(ii)] there exists a smooth curve $C\subset X$ containing $\Gamma$,
such that the induced map $\Jac(C)\to X$ is an isomorphism of \ppav's.
\end{enumerate}
We give $Y_d$ the induced reduced scheme structure.

We shall now prove analogues of (1), (2) and (3) in Theorems \ref{thm:schottky-A} and \ref{thm:schottky-B}
for $Y_d$ with $d\ge 3$:

\begin{prop}\label{prop:Yd}
With $Y_d\subset \Hilb^d(X)$ as defined above, we have:
\begin{enumerate}
\item If $X$ is not a Jacobian, then $Y_d=\emptyset$.
\item If $X\cong \Jac(C)$ for a hyperelliptic curve $C$, then
$Y_d$ is isomorphic to the curve $C$.
\item If $X\cong \Jac(C)$ for a non-hyperelliptic curve $C$, then $Y_d$ is isomorphic to a disjoint union
of two copies of $C$.
\end{enumerate}
\end{prop}

As preparation for the proof, consider a Jacobian $X=\Jac(C)$ for some smooth curve $C$ of genus $g$.
It is convenient to fix an Abel--Jacobi embedding
$C\into X$; any other curve $C'\subset X$
for which $\Jac(C') \to X$ is an isomorphism is of the form
$\pm C_x$ for some $x\in X$.
Such a curve $\pm C_x$ contains the origin $0\in X$ if and only
if $x\in C$.
Hence $Y_d$ is the image of the map
\begin{equation*}
\phi = \phi_+ {\textstyle\coprod} \phi_-\colon C {\textstyle\coprod} C \to \Hilb^d(X)
\end{equation*}
that sends $x\in C$ to the unique degree $d$ subscheme $\Gamma\subset \pm C_x$
supported at $0$, with the positive sign on the first copy
of $C$ and the negative sign on the second copy.

More precisely, $\phi$ is defined as a morphism of schemes as follows.
Let $m\colon X\times X \to X$ denote the group law, and consider
\begin{equation*}
m^{-1}(C) \cap (C\times X)
\end{equation*}
as a family over $C$ via first projection. The fibre over $p\in C$ is 
$C_p$.
Let $N_d = V(\mathfrak{m}_0^d)$ be the $d-1$'st order
infinitesimal neighbourhood of the origin in $X$. Then
\begin{equation*}
Z = m^{-1}(C) \cap (C\times N_d) \subset C\times X
\end{equation*}
is a $C$-flat family of degree $d$ subschemes in $X$;
its fibre over $p\in C$ is $C_p\cap N_d$.
This family defines $\phi_+\colon C \to \Hilb^d(X)$,
and we let $\phi_- = - \phi_+$ (where the minus sign denotes the automorphism of $\Hilb^d(X)$
induced by the group inverse in $X$).

\begin{lem}\label{lem:phi}
The map $\phi_+\colon C\to\Hilb^d(X)$ is a closed embedding for $d>2$.
\end{lem}

In the proof of the Lemma, we shall make use of the difference map $\delta\colon C\times C \to X$, sending a pair $(p,q)$
to the degree zero divisor $p-q$. We let $C-C\subset X$ denote its image.
If $C$ is hyperelliptic, we may and will choose the Abel--Jacobi embedding $C\subset X$ such
that the involution $-1$ on $X$ restricts to the hyperelliptic involution $\iota$ on $C$.
Thus, when $C$ is hyperelliptic, $C-C$ coincides with the distinguished surface $W_2$, and the difference map $\delta$
can be factored via the symmetric product $C^{(2)}$:
\begin{equation*}
\begin{diagram}
C\times C  & \rTo^{1\times \iota}_{\iso} & C\times C \\
\dTo           &          & \dTo^\delta\\
C^{(2)}           & \rTo                 & X
\end{diagram}
\end{equation*}
We note that the double cover $C\times C \to C^{(2)}$,
that sends an ordered pair to the corresponding unordered pair, is branched along the diagonal,
so that via $1\times \iota$, the branching divisor becomes the ``antidiagonal'' $(1,\iota)\colon C\into C\times C$.

As is well known, the surface $C-C$ is singular at $0$, and nonsingular everywhere else.
The blowup of $C-C$ at $0$ coincides with $\delta\colon C\times C\to C-C$ when $C$ is nonhyperelliptic,
and with the addition map $C^{(2)} \to W_2$ when $C$ is hyperelliptic.

\begin{proof}[Proof of Lemma \ref{lem:phi}]
To prove that $\phi_+$ is a closed embedding, we need to show that its restriction to any
finite subscheme $T\subset C$ of degree $2$ is nonconstant, i.e.~that the family
$\res{Z}{T}$ is not a product $T\times \Gamma$. For this it suffices to prove that if
$\Gamma$ is a finite scheme such that
\begin{equation}\label{eq:constant}
m^{-1}(C) \supset T\times \Gamma,
\end{equation}
then the degree of $\Gamma$ is at most $2$.

Consider the following commutative diagram:
\begin{equation}\label{eq:d-diagram}
\begin{diagram}
X\times X & \rTo^{(m, \mathit{pr}_2)}_{\iso} &  X\times X \\
\uInto && \uInto \\
m^{-1}(C) \cap (X\times C)  & \rTo_{\iso} & C\times C\\
          & \rdTo_{\mathit{pr}_1}     &  \dTo^{\delta} \\
          &                           &  X
\end{diagram}
\end{equation}
First suppose $T = \{p,q\}$ with $p\ne q$. The claim is then simply that $C_p\cap C_q$, or equivalently its translate
$C\cap C_{q-p}$, is at most a finite scheme of degree $2$.
Diagram \eqref{eq:d-diagram} identifies the fibre $\delta^{-1}(q-p)$ on the right with precisely $C\cap C_{q-p}$
on the left. But $\delta^{-1}(q-p)$ is a point when $C$ is nonhyperelliptic, and two points if $C$ is hyperelliptic.

Next suppose $T\subset C$ is a nonreduced degree $2$ subscheme supported in $p$. 
Assuming $\Gamma$ satisfies \eqref{eq:constant}, we have $\Gamma\subset C_p$, so
\begin{equation*}
m^{-1}(C) \cap (X\times C_p) \supset T\times \Gamma
\end{equation*}
or equivalently
\begin{equation*}
m^{-1}(C) \cap (X\times C) \supset T_p \times \Gamma_{-p}
\end{equation*}
We have $T_p \subset C-C$, and Diagram \eqref{eq:d-diagram} identifies $\delta^{-1}(T_p)$ on the right
with $m^{-1}(C)\cap(T_p\times C)$ on the left.

Suppose $C$ is nonhyperelliptic.
Then $\delta$ is the blowup of $0\in C-C$, and
$\delta^{-1}(T_p)$ is the diagonal $\Delta_C\subset C\times C$ together with an
embedded point of multiplicity $1$ (corresponding to the tangent direction of $T_p\subset C - C$).
Diagram \eqref{eq:d-diagram} identifies the diagonal in $C\times C$ on the right with $\{0\}\times C$ on the left.
Thus $m^{-1}(C)\cap (T_p\times C)$ is $\{0\}\times C\subset X\times C$ with an embedded point.
Equivalently, $m^{-1}(C) \cap (T\times C_p)$ is $\{p\}\times C_p$ with an embedded point, say at $(p,q)$.
This contains no constant family $T\times \Gamma$ except for $\Gamma=\{q\}$, so $\Gamma$ has at most degree $1$.

Next suppose $C$ is hyperelliptic.
We claim that $\delta^{-1}(T_p)$ is the diagonal $\Delta_C\subset C\times C$ with either two embedded points
of multiplicity $1$, or one embedded point of multiplicity $2$. As in the previous case, this implies
that $m^{-1}(C) \cap (T\times C_p)$ is $\{p\}\times C_p$ with two embedded points of multiplicity $1$ or one
embedded point of multiplicity $2$, and the maximal constant family $T\times \Gamma$ it contains has $\Gamma$ of degree $2$.
It remains to prove that $\delta^{-1}(T_p)$ is as claimed.

We have $W_2=C-C$, and the blowup at $0$ is $C^{(2)}\to W_2=C-C$.
The preimage of $T_p$ is the curve $(1+\iota)\colon C\to C^{(2)}$,
together with an embedded point of multiplicity $1$, say supported at $q+\iota(q)$.
Now the two to one cover $C\times C\to C^{(2)}$ is branched along the diagonal $2C\subset C^{(2)}$,
If $q\ne \iota(q)$, then the preimage in $C\times C$
is just $(1,\iota)\colon C\to C\times C$, together with two embedded points of multiplicity $1$,
supported at $(q, \iota(q))$ and $(\iota(q), q)$.
If $q=\iota(q)$, i.e.~$q$ is Weierstraß, then we claim the preimage in $C\times C$
is $(1,\iota)\colon C\to C\times C$ together with an embedded point of multiplicity $2$.
This follows once we know that the curves $2C$ and $(1+\iota)(C)$ in $C^{(2)}$ intersect transversally.
And they do, as the tangent spaces of the two curves $(1,1)(C)$ (the diagonal) and $(1,\iota)(C)$ in $C\times C$
are invariant under the involution exchanging the two factors, with eigenvalues $1$ and $-1$, respectively.
\end{proof}

\begin{proof}[Proof of Proposition \ref{prop:Yd}]
Point (1) is obvious, so we may assume $X=\Jac C$. By Lemma \ref{lem:phi},
$\phi_+$ is a closed embedding and hence so is $\phi_- = -\phi_+$.
If $C$ is hyperelliptic, we have chosen the embedding $C\subset X$ such that the involution $-1$ on $X$ extends
the hyperelliptic involusion $\iota$ on $C$. It follows that $C_p = -C_{\iota(p)}$, and thus $\phi_- = \phi_+\circ \iota$.
Thus the two maps $\phi_+$ and $\phi_-$ have coinciding image, and (2) follows.

For (3), it remains to prove that if $C$ is nonhyperelliptic, then the images of $\phi_-$ and $\phi_+$ are disjoint,
i.e.~we never have $C_p\cap N_d = (-C_q)\cap N_d$ for distinct points $p, q\in C$.
In fact, $C_p\cap (-C_q)$ is at most a finite scheme of degree $2$: the addition map
\begin{equation*}
C\times C \to X
\end{equation*}
is a degree two branched cover of $C^{(2)} \iso W_2$ (using that $C$ is nonhyperelliptic),
and its fibre over $p+q \in W_2$ is isomorphic to $C_p\cap (-C_q)$.
\end{proof}

\section{Proof of Theorem \ref{thm:schottky-A}}

In view of Proposition \ref{prop:Yd}, it suffices to prove that $Y$ in Theorem
\ref{thm:schottky-A} agrees with $Y_3$ in Proposition \ref{prop:Yd}. This is a reformulation of the Gunning--Welters criterion: given $\Gamma\in \Hilb^3(X)$, consider the set
\begin{equation*}
V_{\Gamma} = \{x\in X \st \text{$\Gamma_X \subset \psi^{-1}(\ell)$ for some line $\ell\subset \PP^{2^{g-1}}$}\}
\end{equation*}
Then Gunning--Welters says that $V_{\Gamma}$ has positive dimension
if and only if $(X, \Theta)$ is a Jacobian.
Moreover, when $V_{\Gamma}$ has positive dimension, it is a smooth curve,
the canonical map $\Jac(V_{\Gamma}) \to X$ is an isomorphism, and $\Gamma$ is contained in $V_{\Gamma}$ (see \cite[Theorem~(0.4)]{Wcriterion}).
Thus $Y$ in Theorem \ref{thm:schottky-A} agrees with $Y_3$ in Proposition \ref{prop:Yd}.

\section{Proof of Theorem \ref{thm:schottky-B}}

Let $X$ be the Jacobian of $C$.
For convencience, we fix an Abel--Jacobi embedding $C\into X$.
First, we shall analyse theta-genericity for finite subschemes of $C$.

Recall the notion of \emph{theta-duality}: whenever $V\subset X$ is a
closed subscheme, we let
\begin{equation*}
T(V) = \{x\in X \st V \subset \Theta_x\}.
\end{equation*}
It has a natural structure as a closed subscheme of $X$ (see \cite[Section 4]{PPminimal} and \cite[Section
2.2]{GL}); the definition as a (closed) subset is sufficient for our present
purpose.

With this notation, theta-genericity means that for all chains of subschemes
\begin{equation}\label{eq:chain}
\Gamma_1\subset \Gamma_2\subset \cdots \subset \Gamma_{g+1}\subset \Gamma,
\end{equation}
where $\Gamma_i$ has degree $i$, the corresponding chain of theta-duals,
\begin{equation*}
T(\Gamma_1)\supset T(\Gamma_2) \supset \cdots \supset T(\Gamma_{g+1}),
\end{equation*}
consists of strict inclusions of sets.

We write $\widehat{\sheaf{F}}$ for the Fourier--Mukai transform \cite{mukai, huybrechts}
of a WIT-sheaf $\sheaf{F}$ on $X$ \cite[Def.~2.3]{mukai}: $\widehat{\sheaf{F}}$ is a sheaf on the dual abelian variety,
which we will identify with $X$ using the principal polarization.

\begin{prop}\label{prop:weierstrass}
Let $\Gamma\subset C$ be a finite subscheme of degree at least $g+1$.
Then $\Gamma$ is theta-general, as a subscheme of $\Jac(C)$, if and only if
$\dim H^0(\O_C(\Gamma_g)) = 1$ for every degree $g$ subscheme $\Gamma_g\subset \Gamma$.
In particular, if $\Gamma$ is supported at a single point $p\in C$,
then $\Gamma$ is theta-general if and only if $p$ is not a Weierstra\ss \ point.
\end{prop}

\begin{proof}
For the last claim, note that the condition $\dim H^0(\O_C(gp)) > 1$ says precisely that $p$
is a Weierstra\ss \ point.

For any effective divisor $\Gamma_g \subset C$ degree $g$, it is well known that $\dim H^0(\O_C(\Gamma_g)) = 1$
if and only if $\Gamma_g$ can be written as the intersection of $C \subset \Jac(C)$ and a $\Theta$-translate
(this is one formulation of Jacobi inversion).
If this is the case, then the point $x\in X$ satisfying $\Gamma_g = C\cap \Theta_x$ is unique.

Consider a chain \eqref{eq:chain}.
If there is a degree $g$ subscheme $\Gamma_g\subset \Gamma$ not of the
form $C\cap \Theta_x$, then every $\Theta$-translate containing $\Gamma_g$ also contains $C$,
and in particular $T(\Gamma_g) = T(\Gamma_{g+1})$. Hence $\Gamma$ is not theta-general.

Suppose, on the other hand, that $\Gamma_g$
is of the form $C\cap \Theta_x$. Then $T(\Gamma_g)\setminus T(\Gamma_{g+1})$ consists (as a set)
of exactly the point $x$.
Thus there is a Zariski open neighbourhood $U\subset X$ of $x$ such that $T(\Gamma_g) \cap U = \{x\}$.
We claim that, for a possibly smaller neighbourhood $U$, there are regular functions
$f_1,\dots, f_g \in \O_X(U)$, such that $T(\Gamma_i) \cap U = V(f_1,\dots, f_i)$ for all $i$:
in fact, apply the Fourier--Mukai functor to the short exact sequence
\begin{equation*}
0\to \sheaf{I}_{\Gamma_i}(\Theta) \to \O_X(\Theta) \to \O_{\Gamma_i} \to 0
\end{equation*}
to obtain
\begin{equation*}
0 \to \O_X(-\Theta) \xrightarrow{F_i} \widehat{\O_{\Gamma_i}} \to \widehat{\sheaf{I}_{\Gamma_i}(\Theta)} \to 0.
\end{equation*}
Then $F_i$ is a section of a locally free sheaf of rank $i$, and its vanishing locus is exactly $T(\Gamma_i)$.
Choose trivializations of $\widehat{\O_{\Gamma_i}}$ over $U$ for all $i$ compatibly, in the sense that the
surjections $\widehat{\O_{\Gamma_{i+1}}} \to \widehat{\O_{\Gamma_i}}$ correspond to projection to the first $i$ factors.
Then $F_i = (f_1,\dots, f_i)$ in these trivializations.

As $T(\Gamma_g)\cap U$ is zero dimensional,
it follows that each $T(\Gamma_i)\cap U$ has codimension $i$ in $U$.
Hence all the inclusions $T(\Gamma_i) \supset T(\Gamma_{i+1})$ are strict, and so $\Gamma$ is
theta-general.
\end{proof}

Now we can compare the locus $Y$ in Theorem \ref{thm:schottky-B} with $Y_{g+2}$ in Proposition \ref{prop:Yd}
by means of the Castelnuovo--Schottky theorem:

\begin{cor}[of Theorem \ref{thm:castelnuovo-schottky}]
Let $\Gamma\in\Hilb^{g+2}(X)$ be theta-general and supported at $0\in X$.
Then $\Gamma$ in the locus $Y_{g+2}$ in Propsition \ref{prop:Yd} if and only
if it is in special position with respect to $2\Theta$-translates.
\end{cor}

\begin{proof}
Theorem \ref{thm:castelnuovo-schottky} immediately shows that if $\Gamma$ is
in special position with respect to $2\Theta$-translates,
then $\Gamma\in Y_{g+2}$.

The converse is straight forward, and does not require the theta-genericity assumption.
Indeed, we use that any curve $C'\subset X$ for which $\Jac(C') \to X$ is an isomorphism is of the form
$\pm C_p$ for some $p\in X$ and we claim that if $\Gamma \subset \pm C_p$, then $\Gamma$ is in special position with respect to $2\Theta$-translates.
For ease of notation, we rename $\pm C_p$ as $C$, so that $\Gamma\subset C$.
Then $H^0(\sheaf{I}_C(2\Theta_x)) \subset H^0(\sheaf{I}_\Gamma(2\Theta_x))$, and
the exact sequence
\begin{equation*}
0 \to
H^0(\sheaf{I}_C(2\Theta_x))\to
H^0(\O_X(2\Theta_x))\to
H^0(\O_C(2\Theta_x))
\end{equation*}
shows that already the codimension of $H^0(\sheaf{I}_C(2\Theta_x))$ in $H^0(\O_X(2\Theta_x))$
is at most $\dim H^0(\O_C(2\Theta_x)) = g+1$.
\end{proof}

Theorem \ref{thm:schottky-B} now follows: 
The set $Y$ defined there agrees with the theta-general elements in $Y_{g+2}$, by the Corollary.
By Proposition \ref{prop:weierstrass},
$\Gamma = \phi_{\pm}(p)$ is theta-general if and only if the supporting point $0$
of $\Gamma$ is not Weierstra\ss{} in $\pm C_p$, i.e.\ $p\in C$
is not Weierstra\ss.

\section{Historical remark}

Assume $C$ is not hyperelliptic.
Then $C_p\cap(-C_p)$ is a finite subscheme of degree $2$ supported at $0$.
Thus, for $d=2$, we have $\phi_+=\phi_-$, and the argument in Lemma \ref{lem:phi}
shows that $\phi_+$ is an isomorphism from $C$ onto $Y_2$.
If $C$ is hyperelliptic with hyperelliptic involution $\iota$, however, we find
that $\phi_+$ factors through $C/\iota \iso \PP^1$, and $Y_2 \iso \PP^1$,
and we cannot reconstruct $C$ from $Y_2$ alone.

In the nonhyperelliptic situation, it is well known that the curve $C$ can be reconstructed
as the projectivized tangent cone to the surface $C-C\subset X$ at $0$.
This projectivized tangent cone is exactly $Y_2$ (when we identify the projectivized tangent space
to $X$ at $0$ with the closed subset of $\Hilb^2(X)$ consisting of nonreduced degree $2$ subschemes supported at $0$).
To quote Mumford
\cite{Mcurves}: ``If $C$ is hyperelliptic, other arguments are needed.'' In
the present note, these other arguments are to increase $d$!

\bibliographystyle{plain}
\bibliography{hilbjac}

\end{document}